\documentclass[a4paper]{article}

\usepackage{a4wide}
\usepackage[english]{babel} 
\usepackage{amsmath,amssymb}
\usepackage[amsmath,thmmarks]{ntheorem}
\usepackage{booktabs}
\usepackage[final]{microtype}

\theoremstyle{plain}
\theoremheaderfont{\normalfont\bfseries}
\theorembodyfont{\itshape}
\theoremseparator{.}
\theoremnumbering{arabic}
\theoremsymbol{}
\newtheorem{thm}{Theorem}

\newtheorem{lem}[thm]{Lemma}

\theoremstyle{plain}
\theoremheaderfont{\normalfont\bfseries}
\theorembodyfont{\normalfont}
\theoremseparator{.}
\theoremnumbering{arabic}
\theoremsymbol{}
\newtheorem{defn}[thm]{Definition}

\theoremstyle{plain}
\theoremheaderfont{\normalfont\bfseries}
\theorembodyfont{\normalfont}
\theoremseparator{.}
\theoremnumbering{arabic}
\theoremsymbol{\ensuremath{\lozenge}}
\newtheorem{ex}[thm]{Example}

\theoremstyle{nonumberplain}
\theoremheaderfont{\itshape}
\theorembodyfont{\normalfont}
\theoremseparator{.}
\theoremnumbering{arabic}
\theoremsymbol{\ensuremath{\square}}
\newtheorem{proof}{Proof}

\theoremstyle{empty}
\theoremheaderfont{\itshape}
\theorembodyfont{\normalfont}
\theoremseparator{.}
\theoremnumbering{arabic}
\theoremsymbol{\ensuremath{\square}}

\newcommand{\Z}{\mathbb{Z}}
\newcommand{\Q}{\mathbb{Q}}
\newcommand{\R}{\mathbb{R}}
\newcommand{\C}{\mathbb{C}}

\newcommand{\gauss}{\mbox{$_2 F_1$}}

\newcommand{\A}{\mathcal{A}}
\newcommand{\HA}{H_\A(\bbeta)}

\newcommand{\trian}{\mathcal{T}}
\newcommand{\lat}{\mathbb{L}}
\newcommand{\inv}[1]{\frac{1}{#1}}
\newcommand{\half}{\inv{2}}

\newcommand{\subs}{\subseteq}

\newcommand{\poch}[2]{(#1)_{#2}}
\newcommand{\vol}{\tr{Vol}}

\newcommand{\ld}{\ldots}
\newcommand{\tr}{\textrm}

\newcommand{\vect}[1]{\boldsymbol{#1}}
\newcommand{\va}{\vect{a}}
\newcommand{\vb}{\vect{b}}
\newcommand{\ve}{\vect{e}}
\newcommand{\vl}{\vect{l}}

\newcommand{\vz}{\vect{z}}

\newcommand{\bbeta}{\vect{\beta}}
\newcommand{\bgam}{\vect{\gamma}}

\newcommand{\ahypergeometric}{$\A$-hyper\-geo\-me\-tric }

\newlength{\eerstekollauricella}
\newlength{\tweedekollauricella}
\newlength{\eerstekolhorn}
\newlength{\tweedekolhorn}
\title{Explicit formulas for parametrized families of irreducible algebraic Appell, Lauricella and Horn functions}
\author{Esther Bod\thanks{esther.bod@gmail.com. Department of Mathematics, Utrecht University, P.O.\ Box 80010, 3508 TA Utrecht, The Netherlands.
This work was supported by the Netherlands Organisation for Scientific Research (NWO) by grant number OND1331860.}}
\date{\today}

\begin{document}

\maketitle

\begin{abstract}
For suitable choices of the parameters the Appell, Lauricella and Horn functions are algebraic.
In~\cite{bod_algebraic_appell_horn} it is shown that there exist families of algebraic functions for the $F_4$, $F_C$, $G_3$, $H_4$, $H_5$ and $H_7$ functions, depending on one parameter $r$.
For each of these we give a solution of the form $f^r g$ of the system of differential equations and give formulas for $f$ and $g$.
\end{abstract}

The hypergeometric Gauss function is defined by 
\begin{equation*}
\gauss(a,b,c|z) = \sum_{n \geq 0} \frac{\poch{a}{n} \poch{b}{n}}{\poch{c}{n} n!} z^n.
\end{equation*}
In 1873 Schwarz determined the parameters $(a,b,c)$ such that $\gauss(a,b,c|z)$ is irreducible and algebraic~\cite{schwarz_gaussfunction}.
Up to translations of $(a,b,c)$ over $\Z^3$ there are three 1-parameter families of irreducible algebraic Gauss functions: 
$\gauss(r,-r,\half|z)$, $\gauss(r,r+\half,\half|z)$ and $\gauss(r,r+\half,2r|z)$ with $r \in \Q \setminus \half \Z$.
The following formulas are well-known:
\begin{align}\label{eq:formulas_gauss} 
\gauss\left(r,-r,\half|z\right) & = \half \left( (\sqrt{1-z}+i\sqrt{z})^{2r} + (\sqrt{1-z} - i\sqrt{z})^{2r} \right), \notag \\ 
\gauss\left(r,r+\half,\half|z\right) & = \half \left( (1+\sqrt{z})^{-2r} + (1-\sqrt{z})^{-2r} \right) \qquad \tr{and} \\
\gauss\left(r,r+\half,2r|z\right) & = \frac{(1+\sqrt{1-z})^{1-2r}}{2^{1-2r} \sqrt{1-z}}. \notag
\end{align}
It follows from these formulas that the monodromy group is the dihedral group with $2q$ elements, where $q$ is the denominator of $2r$.
The Appell, Lauricella and Horn functions are multi-variable generalizations of the Gauss function~\cite{appell_series_hypergeometric_1880,appell_fonctions_hypergeometric_1882,horn_convergenz_hypergeometrische_reihen_1889,lauricella_funcioni_ipergeometriche_piu_variabili,horn_hypergeometrische_funktionen_1931}.
The parameters such that these functions are irreducible and algebraic are computed in~\cite{bod_algebraic_appell_horn}.
Similar to the Gauss function, there are 1-parameter families of irreducible algebraic functions for the Appell $F_4$, the Lauricella $F_C$ and the Horn $G_3$, $H_1$, $H_4$, $H_5$ and $H_7$ functions.
The $H_7$ function is isomorphic to $H_4$, so we will not consider it in this paper.
Furthermore, some of the functions are symmetric with respect to permutations of the parameters.
In all cases, algebraicity only depends on the fractional part of the parameters. 
Up to these permutations and equivalence of the parameters modulo $\Z$, the families of algebraic functions are $F_C(r,-r,\half,\ld,\half|\vect{z})$, $F_C(r,r+\half,\half,\ld,\half|\vect{z})$ and $F_C(r,r+\half,\half,\ld,\half,2r|\vect{z})$ and their 2-variable counterparts $F_4(r,-r,\half,\half|x,y)$, $F_4(r,r+\half,\half,\half|x,y)$ and $F_4(r,r+\half,\half,2r|x,y)$, as well as $G_3(r,-r|x,y)$, $H_4(r,-r,\half,\half|x,y)$, $H_4(r,r+\half,\half,\half|x,y)$, $H_4(r,r+\half,\half,2r|x,y)$ and $H_5(r,-r,\half|x,y)$ with $r \in \Q \setminus \half \Z$.  
The goal of this paper is to give formulas similar to~\eqref{eq:formulas_gauss} for these families. \\

Our results are inspired by an observation by Beukers.
In~\cite{beukers_algebraic_ahypergeometric_functions} he notes that the function $G_3(r,1-r|x,y)$ is of a very special form: 
\begin{equation}\label{eq:g3equation}
G_3(r,1-r|x,y) = f(x,y)^r \cdot \sqrt{\frac{g(x,y)}{\Delta(x,y)}}
\end{equation}
where $f$ and $g$ generate the same cubic extension of $\C[x,y]$ and $\Delta$ is the discriminant of the minimal polynomial of $f$.
More explicitly:
\begin{gather}
y f(x,y)^3 + f(x,y)^2 - f(x,y) - x = 0, \notag \\ 
\qquad \qquad g(x,y) = - 3y^2 f(x,y)^2 - 2y f(x,y) + 4y + 1 \qquad \tr{and} \label{eq:g3equations_fg} \\ 
\Delta(x,y) = 1+4x+4y+18x y - 27 x^2 y^2. \notag 
\end{gather}
In particular, $G_3(r,1-r|x,y)$ only depends on $r$ in the exponent of $f$. \\

By computing $\frac{\Phi(r|\vect{z})^2}{\Phi(2r|\vect{z})}$ for several values of $r$, we see that the other families of algebraic Appell, Lauricella and Horn functions are not of the form 
\begin{equation}\label{eq:function_power}
\Phi(r | \vect{z}) = f(\vect{z})^r g(\vect{z}). 
\end{equation}
Therefore, we will consider the systems of differential equations that these functions satisfy and look for other solutions of the system that are indeed of the form~\eqref{eq:function_power}, with $f$ and $g$ generating the same extension of $\C(\vect{z})$.
Usually the Appell and Horn hypergeometric system are viewed as solutions of systems of two second order differential equations.
However, we will consider \ahypergeometric functions, since this makes it easier to give a basis for the solution space, as well as to check the differential equations for the Lauricella $F_C$ function.

We first introduce \ahypergeometric systems of differential equations and explain their relationship to the classical hypergeometric functions. 
Then we indicate how a basis of the solution space can be found using triangulations of the convex hull of $\A$.
The strategy to find $\Phi$ is as follows:
if there exists a solution $\Phi$ of the form~\eqref{eq:function_power}, then it can be written on the basis of the solution space.
We make an educated guess about the coefficients of $\Phi$ on this basis and the formulas for $f$ and $g$.
To prove that these guesses are correct, we change our perspective and define $f$ and $g$ by the formulas we just found, and define $\Phi$ to be the linear combination of the element of the basis of the solution space given by the coefficients we guessed.
For each of the Appell and Horn functions, we give formulas for $f$ and $g$ and show that $f^r g$ is indeed a solution of the system of differential equations.
Furthermore, we show that $\Phi$ is indeed equal to $f^r g$.
For the Lauricella functions, we don't explicitly compute the basis, but we do give equations for $f$ and $g$ and show that $f^r g$ is a solution of the system of differential equations. \\

\ahypergeometric functions were introduced in the 1980's by Gelfand, Graev, Kapranov and Zelevinsky in a series of papers~\cite{ggz_holonomic_systems_equations_series_hypergeometric_type,gzk_equations_hypergeometric_newton_polyhedra,gzk_hypergeometric_functions_and_toral_manifolds,gkz_hypergeometric_functions_and_toral_manifolds_correction}.
They are defined as follows: 

\begin{defn}\label{defn:Ahypergeometric_function}
Let $\A=\{\va_1, \ld, \va_N\}$ be a finite subset of $\Z^r$ such that $\Z \A = \Z^r$ and there exists a linear form $h$ on $\R^r$ such that $h(\va_i)=1$ for all $i$.
The \emph{lattice of relations} of $\A$ is $\lat = \{(l_1, \ld, l_N) \in \Z^N \ | \ l_1 \va_1 + \ld + l_N \va_N = 0 \}$.
Let $\bbeta \in \C^r$ and denote by $\partial_i$ the differential operator $\frac{\partial}{\partial z_i}$.
The \emph{\ahypergeometric system associated to $\A$ and $\bbeta$}, denoted $\HA$, consists of two sets of differential equations:
\begin{itemize}
\item
the \emph{structure equations}: for all $\vl = (l_1, \ld, l_N) \in \lat$ 
\begin{equation*}
\square_{\vl} \Phi = 
\left( \prod_{l_i>0} \partial_i^{l_i} \right) \Phi - \left( \prod_{l_i<0} \partial_i^{-l_i} \right) \Phi = 0. 
\end{equation*}
\item 
the \emph{homogeneity} or \emph{Euler equations}: for $1 \leq i \leq r$ 
\begin{equation*}
a_{1i} z_1 \partial_1 \Phi + \ld + a_{Ni} z_N \partial_N \Phi = \beta_i \Phi.
\end{equation*}
\end{itemize}
The solutions $\Phi(z_1, \ld, z_N)$ of this system are called \emph{\ahypergeometric functions}.
\end{defn}

The connection between the Appell, Lauricella and Horn hypergeometric functions and systems of \ahypergeometric differential equations is given by $\Gamma$-series~\cite{gzk_hypergeometric_functions_and_toral_manifolds}.
Let $\bgam \in \C^N$ such that $\gamma_1 \va_1 + \ld + \gamma_N \va_N = \bbeta$.
Then it can easily be checked that 
\begin{equation}\label{eq:gamma_series}
\Phi_{\bgam}(z_1, \ld, z_N) = \sum_{(l_1, \ld, l_N) \in \lat} \frac{z_1^{l_1+\gamma_1} \cdot \ld \cdot z_n^{l_N+\gamma_N}}{\Gamma(l_1+\gamma_1+1) \cdot \ld \cdot \Gamma(l_N+\gamma_N+1)}
\end{equation}
is a formal solution of $\HA$.
This function has $N$ variables, but up to the monomial factor $z_1^{\gamma_1} \cdots z_N^{\gamma_N}$ it can be viewed as a function in $d=N-r$ variables:
the lattice $\lat$ has rank $d$ and hence has a basis $\{\vb_1, \ld, \vb_d\} \subs \Z^N$.
Now $\vz^{\vb_1}, \ld, \vz^{\vb_d}$ can be used as new variables.
We will call this function in $d$ variables the \emph{dehomogenization} of $\Phi_{\bgam}$. 
The Appell, Lauricella and Horn functions can be viewed as dehomogenizations of suitable \ahypergeometric functions.
For example, with $\A = \{\ve_1, \ve_2, \ve_3, \ve_4, \ve_1+\ve_2-\ve_3, \ve_1+\ve_2-\ve_4\}$ we get $\lat = (-1, -1, 1, 0, 1, 0)\Z \oplus (-1, -1, 0, 1, 0, 1)\Z$.
Choose $\bbeta = (-a, -b, c_1-1, c_2-1)$ and $\bgam = (-a, -b, c_1-1, c_2-1, 0, 0)$ to obtain the $\Gamma$-series
\begin{multline*}
\Phi_{\bgam}(z_1, \ld, z_6) = z_1^{-a} z_2^{-b} z_3^{c_1-1} z_4^{c_2-1} \,\,\cdot \\
\sum_{m,n \in \Z} \frac{\left( \frac{z_3 z_5}{z_1 z_2} \right)^m \left( \frac{z_4 z_6}{z_1 z_2} \right)^n}{\Gamma(1-a-m-n) \Gamma(1-b-m-n) \Gamma(m+c_1) \Gamma(n+c_2) \Gamma(m+1) \Gamma(n+1)}.
\end{multline*}
Up to a constant factor, dehomogenization gives the Appel $F_4$ function 
\begin{equation*}
F_4(a,b,c_1,c_2|x,y) = \sum_{m,n \geq 0} \frac{\poch{a}{m+n} \poch{b}{m+n}}{\poch{c_1}{m} \poch{c_2}{n} m! n!} x^m y^n.
\end{equation*}

Gelfand, Kapranov and Zelevinsky showed that there is a one to one correspondence between regular triangulations of the convex hull of $\A$ and locally converging bases of $\Gamma$-series solutions~\cite{gzk_hypergeometric_functions_and_toral_manifolds}.
In many interesting cases, the dimension of the solution space of $\HA$ can be computed from the combinatorics of $\A$:

\begin{defn}\label{defn:normality}
$\A$ is called \emph{normal} if the integral points in the non-negative cone spanned by $\A$ are integral non-negative combinations of the vectors $\va_i$.
\end{defn}

\begin{thm}[Gelfand, Kapranov, Zelevinsky, Adolphson]\label{thm:dim_sol}
Let $\A$ and $\bbeta$ as in Definition~\ref{defn:Ahypergeometric_function}.
Then $\HA$ is holonomic.
If $\A$ is normal, then the dimension of the solution space is equal to the simplex volume of the convex hull of $\A$.
\end{thm}

By simplex volume we mean the volume that is normalized so that an elementary simplex has volume 1.
One can easily show that $\A$ is normal for all Appell, Lauricella and Horn functions.

The correspondence between regular triangulations and local bases is as follows.
Let $\trian$ be a regular triangulation.
Each simplex in $\trian$ is the convex hull of $r$ elements of $\A$, say $\{\va_i \ | \ i \in I\}$.
We write $\trian = \{I_1, \ld, I_n\}$ if the simplices have vertices $\{\va_i \ | \ i \in I_j\}$ for $1 \leq j \leq n$.
Any $\bgam \in \C^N$ such that $\gamma_j \in \Z$ if $j \not\in I$ gives a convergent $\Gamma$-series $\Phi_{\bgam}$.
It can be shown that there are $\det((\va_i)_{i \in I})$ choices for $\bgam$ such that the resulting $\Gamma$-series are distinct.
In case $\HA$ is totally non-resonant, i.e., $\bbeta+\Z^r$ contains no point in any hyperplane spanned by $r-1$ independent elements of $\A$, these series are linearly independent.
The series coming from differents simplices are also independent, and hence these series form a basis of the solution space of $\HA$. \\

We now prove some of the formulas~\eqref{eq:g3equation} for $G_3(r,1-r|x,y)$ and similar formulas for $F_4(r,r+\half,\half,\half|x,y)$ and $F_C(r,r+\half,\half,\ld,\half|\vect{z})$.
As the proofs for the other functions are similar, we omit them here, but the results can be found in Tables~\ref{tab:F41_formula} up to~\ref{tab:H5_formula}.
The results for the Appell and Horn functions are stated in terms of the classical Horn systems.
These systems are described in~\cite{dickenstein_hypergeometric_functions_binomials} and consist of two second order partial differential equations.
These equations can also be found in the tables, using the notation $\theta_x = x \frac{\partial}{\partial x}$ and $\theta_y = y \frac{\partial}{\partial y}$.
The Horn systems can have more independent solutions than the \ahypergeometric systems: apart from the $\Gamma$-series solutions, there can be monomial solutions.
By~\cite[Theorem~2.5]{dickenstein_matusevich_sadykov_bivariate_hypergeometric_dmodules}, the rank of all Horn systems we consider is 4.
For the $F_4$, $H_4$ and $H_5$ functions, this equals the rank of the \ahypergeometric system.
However, for $G_3(r,1-r|x,y)$ the Horn system has a so called Puiseux monomial solution $x^{\frac{r-2}{3}} y^{\frac{-r-1}{3}}$ that is not a solution of the \ahypergeometric system. \\ 

\begin{lem}\label{lem:f4proof_equation}
Let $\Phi(r|x,y) = (\sqrt{x}+\sqrt{y}-1)^{-2r}$.
Then $\Phi(r|x,y)$ is a solution of the system of Horn equations for $F_4(r,r+\half,\half,\half|x,y)$.
Furthermore, on the basis $\{\Phi_1, \Phi_2, \Phi_3, \Phi_4\}$ of the solution space as in Table~\ref{tab:F42_formula}, we have $\Phi = \Phi_1 + 2r \Phi_2 + 2r \Phi_3 + 2r(2r+1) \Phi_4$.
\end{lem}

\begin{proof}
It can easily be checked that $\Phi(r|x,y)$ is a solution of the Horn system of differential equations from Table~\ref{tab:F42_formula}.

It follows from~\cite[Theorem~2.5]{dickenstein_matusevich_sadykov_bivariate_hypergeometric_dmodules} that the rank of the Horn system is the rank $\vol(\A)=4$ of the \ahypergeometric system.
We choose the triangulation $\{\{1, 2, 3, 4\}$, $\{1, 2, 3, 6\}$, $\{1, 2, 4, 5\}$, $\{1, 2, 5, 6\}\}$ of $\A$ and 
the corresponding vectors 
$\bgam_1 = (-r, -r-\half, -\half, -\half, 0, 0)$, 
$\bgam_2 = (-r-\half, -r-1, -\half, 0, 0, \half)$, 
$\bgam_3 = (-r-\half, -r-1, 0, -\half, \half, 0)$ and 
$\bgam_4 = (-r-1, -r-\frac{3}{2}, 0, 0, \half, \half)$.
This gives the dehomogenized basis 
$\Phi_1(r|x,y) = F_4\left(r,r+\half,\half,\half|x,y\right)$, 
$\Phi_2(r|x,y) = \sqrt{y} F_4\left(r+\half, r+1, \half, \frac{3}{2}|x,y\right)$,  
$\Phi_3(r|x,y) = \sqrt{x} F_4\left(r+\half, r+1, \frac{3}{2}, \half|x,y\right)$ and
$\Phi_4(r|x,y) = \sqrt{x y} F_4\left(r+1,r+\frac{3}{2},\frac{3}{2}, \frac{3}{2} |x,y\right)$
of the solution space of $\HA$.
Since $\Phi(r|x,y)$ is a solution, there exist functions $c_i(r)$, independent of $x$ and $y$, such that
\begin{equation*}
\Phi(r|x,y) = c_1(r) \Phi_1(r|x,y) + c_2(r) \Phi_2(r|x,y) + c_3(r) \Phi_3(r|x,y) + c_4(r) \Phi_4(r|x,y).
\end{equation*}
The lowest order terms in the Taylor series of $\Phi(r|x,y)$ are
\begin{equation*}
(\sqrt{x}+\sqrt{y}-1)^{-2r} = 1 + 2r \sqrt{x} + 2r \sqrt{y} + (r+2r^2)x + 2r(2r+1) \sqrt{xy} + (r+2r^2) y + \ld 
\end{equation*}
Computing the lowest order terms of $\Phi_1$, $\Phi_2$, $\Phi_3$ and $\Phi_4$ shows that $a_1=1$, $a_2=a_3=2r$ and $a_4=2r(2r+1)$, so indeed $\Phi(r|x,y) = (\sqrt{x}+\sqrt{y}-1)^{-2r}$.
\end{proof}

\begin{lem}\label{lem:fcproof_equation}
$\Phi(r|\vect{z}) = (\sqrt{z_1}+\ld+\sqrt{z_n}-1)^{-2r}$ is a dehomogenized solution of the \ahypergeometric system for $F_C(r,r+\half,\half,\ld,\half|z_1, \ld, z_n)$.
\end{lem}

\begin{proof}
The homogeneous function corresponding to $\Phi$ is 
\begin{equation*}
\tilde{\Phi}(r|\vect{z}) = 
z_1^{-r} z_2^{-r-\half} z_3^{-\half} \cdots z_{n+2}^{-\half} \left(\sqrt{\frac{z_3 z_{n+3}}{z_1 z_2}}+\ld+\sqrt{\frac{z_{n+2} z_{2n+2 }}{z_1 z_2}}-1\right)^{-2r}
\end{equation*}
with $z_i$ in $\Phi$ equal to $\frac{z_{i+2} z_{n+i+2}}{z_1 z_2}$ in $\tilde{\Phi}$.
The system $\HA$ is given by $\A = \{\ve_1, \ve_2, \ld, \ve_{n+2}, \ve_1+\ve_2-\ve_3, \ve_1+\ve_2-\ve_4, \ld, \ve_1+\ve_2-\ve_{n+2}\}$ and $\bbeta = (-r, -r-\half, -\half, \ld, -\half)$.
Hence the lattice is $\lat = \bigoplus_{i=1}^n \Z(-\ve_1-\ve_2+\ve_{i+2}+\ve_{n+i+2}) \subs \Z^{2n+2}$.
One can easily show that the structure equations all follow from 
\begin{equation*}
\partial_1 \partial_2 \tilde{\Phi} =
\partial_3 \partial_{n+3} \tilde{\Phi} =
\ld =
\partial_{n+2} \partial_{2n+2} \tilde{\Phi}.
\end{equation*}
An easy computation shows that both $\partial_1 \partial_2 \tilde{\Phi}$ and $\partial_k \partial_{n+k} \tilde{\Phi}$ equal
\begin{equation*}
z_1^{-r-1} z_2^{-r-\frac{3}{2}} z_3^{-\half} \cdots z_{n+2}^{-\half} \left(\sqrt{\frac{z_3 z_{n+3}}{z_1 z_2}}+\ld+\sqrt{\frac{z_{n+2} z_{2n+2 }}{z_1 z_2}}-1\right)^{-2r-2} 
\end{equation*}
for all $3 \leq k \leq n+2$.
Hence the structure equations are satisfied.
The Euler equations
\begin{gather*}
(z_1 \partial_1 + z_{n+3} \partial_{n+3} + \ld + z_{2n+2} \partial_{2n+2}) \tilde{\Phi}(r|\vect{z}) = -r \tilde{\Phi}(r|\vect{z}) \\
(z_2 \partial_2 + z_{n+3} \partial_{n+3} + \ld + z_{2n+2} \partial_{2n+2}) \tilde{\Phi}(r|\vect{z}) = -(r+\half) \tilde{\Phi}(r|\vect{z}) \\
(z_k \partial_k - z_{n+k} \partial_{n+k}) \tilde{\Phi}(r|\vect{z}) = -\half \tilde{\Phi}(r|\vect{z}) \quad \tr{with } 3 \leq k \leq n+2 
\end{gather*}
are easily checked.
\end{proof}

\begin{lem}\label{lem:g3proof_equation}
$G_3(r,1-r|x,y)$ satisfies formulas~\eqref{eq:g3equation} and~\eqref{eq:g3equations_fg}.
\end{lem}

\begin{proof}
For $G_3(r,1-r|x,y)$, the Horn system of differential equations is
\begin{align}
(x (2\theta_x-\theta_y+1-r)(2\theta_x-\theta_y+2-r) - (-\theta_x+2\theta_y+r)\theta_x) \Phi & = 0 \label{eq:system_G3} \notag \\
(y (-\theta_x+2\theta_y+r)(-\theta_x+2\theta_y+1-r) - (2\theta_x-2\theta_y+1-r)\theta_y) \Phi & = 0 
\end{align}
where $\theta_x = x \frac{\partial}{\partial x}$ and $\theta_y = y \frac{\partial}{\partial y}$.
Define $\Phi(r|x,y) = f(x,y)^r \sqrt{\frac{g(x,y)}{\Delta(x,y)}}$ where $f, g$ and $\Delta$ are as in~\eqref{eq:g3equations_fg}.
We will first show that $\Phi$ is a solution of the Horn system and then show that $\Phi(r|x,y)=G_3(r,1-r|x,y)$. 

To show that $\Phi$ is a solution of the Horn system, we compute the partial derivatives of $f$ and $g$ using implicit differentiation.
For example,
\begin{equation*}
\frac{\partial f}{\partial x} = \frac{1}{-1+2f(x,y)+3yf(x,y)^2}.
\end{equation*}
Substituting this in~\eqref{eq:system_G3} and dividing by $\Phi(r|x,y)$ gives two expressions containing integral powers of $f$ and $g$.
Using $g(x,y) = - 3y^2 f(x,y)^2 - 2y f(x,y) + 4y + 1$, we obtain expressions containing integral powers of 
$f$ only.
As $f$ satisfies an equation of degree 3, these expressions can be reduced to contain only the powers 1, $f$ and $f^2$.
It turns out that these expressions are 0, so $\Phi(r|x,y)$ is a solution of~\eqref{eq:system_G3}.

To show that $\Phi(r|x,y) = G_3(r,1-r|x,y)$, we write $\Phi$ on a basis of the solution space of the Horn system.
However, the Horn system is not equivalent to the \ahypergeometric system.
It follows from~\cite[Theorem~2.5]{dickenstein_matusevich_sadykov_bivariate_hypergeometric_dmodules} that the rank of~\eqref{eq:system_G3} is $\vol(\A)+1 = 4$.
Apart from the three independent solutions of the \ahypergeometric system, there is also a Puiseux monomial solution.
With the triangulation $\{\{1,2\}, \{2,3\}, \{1,4\}\}$ of $\A$, we get the three solutions $\Phi_1(r|x,y) = G_3(r,1-r|x,y)$, 
\begin{align*}
\Phi_2(r|x,y) & = x^r \sum_{m, n \geq 0} \frac{\poch{r+1}{2m+3n}}{\poch{r+1}{m+2n} m! n!} (-x)^m (x^2 y)^n 
\qquad
\tr{and} \\
\Phi_3(r|x,y) & = y^{1-r} \sum_{m, n \geq 0} \frac{\poch{2-r}{3m+2n}}{\poch{2-r}{2m+n} m! n!} (x y^2)^m y^n. 
\end{align*}
The Puiseux monomial solution is $\Phi_4(r|x,y) = x^{\frac{r-2}{3}} y^{\frac{-r-1}{3}}$.
Note that $\Phi(r|x,y)$ is a holomorphic solution of the system~\eqref{eq:system_G3}, so it can be expanded in a series that contains only integral powers of $x$ and $y$.
Since it can be written on the basis mentioned above, it is a linear combination of $\Phi_1$, $\Phi_2$, $\Phi_3$ and $\Phi_4$.
By comparing the local exponents, we see that $\Phi(r|x,y) = \Phi_1(r|x,y) = G_3(r|x,y)$ for all $r \not\in \Z$.
By continuity, this also holds for $r \in \Z$.
\end{proof}

Although the proofs of the above results are elementary, it might not be clear how to find these functions $\Phi$ and the formulas.
We end with some remarks and an example about this.
Let $H_{\A}(\bbeta(r))$ by a family of \ahypergeometric systems depending on a parameter $r$ and suppose that there exists a family $\Phi(r|\vect{z})$ of algebraic solutions of the form~\eqref{eq:function_power}.
We compute the regular triangulations and the corresponding bases for the solution space of $H_{\A}(\bbeta(r))$ and choose a suitable basis $\{\Phi_1(r|\vect{z}), \ld, \Phi_k(r|\vect{z})\}$ (in practice, we choose a basis containing the Appell or Horn function).
Then $\Phi(r|\vect{z})$ is a linear combination of the basis elements and hence can be written as
\begin{equation*}
\Phi(r|\vect{z}) = c_1(r) \Phi_1(r|\vect{z}) + \ld + c_k(r) \Phi_k(r|\vect{z})
\end{equation*}
with unknown coefficients $c_i(r)$ depending on $r$.
Note that for all $r,s$, 
\begin{equation}\label{eq:phirs}
\Phi(r|\vect{z}) \Phi(s|\vect{z}) = f(\vect{z})^{r+s} g(\vect{z})^2 = \Phi\left(\frac{r+s}{2}|\vect{z}\right)^2.
\end{equation}
Substituting values for $r$ and $s$ (for example, $r=0$ and $s=2$) and comparing the first few terms of these power series of the left and right hand side gives a system of quadratic equations for $c_i(r)$, $c_i(s)$ and $c_i(\frac{r+s}{2})$.
We solve this system and guess that $c_i(r)$ is an `easy' function of $r$, for example linear or quadratic.
This gives us a guess for $\Phi$.
As a further check, we compute 
\begin{equation*}
\left(\frac{\Phi(r|x,y)}{\Phi(s|x,y)}\right)^{\inv{r-s}} =
\left(\frac{f(x,y)^r g(x,y)}{f(x,y)^s g(x,y)}\right)^{\inv{r-s}} =
f
\end{equation*}
for several values of $r$ and $s$ and check that the first terms of the power series coincide for all choices of $r$ and $s$.

Having convinced ourselves that $\Phi$ is of the desired form, we try to find the formulas for $f$ and $g$.
First we compute the first terms of the power series expansions of $f$ and $g$.
This can easily be done, for example using $g(x,y)=\Phi(0|x,y)$ and $f(x,y) = \frac{\Phi(1|x,y)}{\Phi(0|x,y)}$.
Finding the formulas for $f$ and $g$ is the hardest part and requires some luck and good guesses.
For some functions, we find $g=1$.
Otherwise, we can use substitutions such as $u=\sqrt{x}$ or $u=\frac{x}{2}$ to remove square roots of $x$ and $y$ and factors $2^m$ coming from Pochhammer symbols $\poch{\half}{m}$ to obtain a power series of which we recognize the coefficients.
The Online Database of Integer Sequences~\cite{oeis} can be helpful in this.
In many cases, substituting $x=0$ or $y=0$ gives a well-known function such as the Gauss function $\gauss$.
This implies that substituting $x=0$ or $y=0$ in the equation for $\Phi$ must give an equation the Gauss function satisfies.

To find formulas for the algebraic Lauricella $F_C$ functions, we simply guess the most obvious generalizations of the formulas for the Appell $F_4$ functions.

We illustrate this with $F_4(r,r+\half,\half,\half|x,y)$.

\begin{ex}\label{ex:f4_equations}
For $F_4(r,r+\half,\half,\half|x,y)$, we have $\A = \{\ve_1, \ve_2, \ve_3, \ve_4, \ve_1+\ve_2-\ve_3, \ve_1+\ve_2-\ve_4\}$ and $\bbeta = (-r, -r-\half, -\half, -\half)$. 
Suppose that $\Phi(r|z)$ is a solution of the form~\eqref{eq:function_power}.
There are coefficients $c_1(r), \ld, c_4(r)$ such that after dehomogenization
\begin{equation*}
\Phi(r|x,y) = c_1(r) \Phi_1(r|x,y) + c_2(r) \Phi_2(r|x,y) + c_3(r) \Phi_3(r|x,y) + c_4(r) \Phi_4(r|x,y).
\end{equation*}
As it suffices to determine $\Phi$ up to a scalar, we simplify by assuming that $c_1(r)=1$ (this need not be possible, if all solutions have $c_1(r)=0$ for some $r$, but we can try it as a first guess).
By taking $r=0$ and $s=2$, applying relation~\eqref{eq:phirs} and comparing the coefficients of the first few powers of $x$ and $y$, we get a system of quadratic equations for the coefficients $c_i(0)$, $c_i(1)$ and $c_i(2)$.
Solving this, we obtain $c_2(0)=c_3(0)=0$, $c_2(1)=c_3(1)=2$, $c_2(2)=c_3(2)=4$, $c_4(0)=0$, $c_4(1)=6$ and $c_4(2)=20$. 
Assuming that that coefficients are easy functions of $r$, we guess that $c_1(r)=1$, $c_2(r)=c_3(r)=2r$ and $c_4(r)=2r(2r+1)$.
Hence we get
\begin{multline*}
\Phi(r|x,y) = F_4(r,r+\half,\half,\half|x,y) + 2r \sqrt{y} F_4(r+\half,r+1,\half,\frac{3}{2}|x,y) + \\
2r \sqrt{x} F_4(r+\half,r+1,\frac{3}{2},\half|x,y) + 2r(2r+1) \sqrt{xy} F_4(r+1,r+\frac{3}{2},\frac{3}{2},\frac{3}{2}|x,y).
\end{multline*}
Substituting $r=0$, we see that $g(x,y)=F_4(0,\half,\half,\half|x,y) = 1$, so $f(x,y) = \Phi(1|x,y)$.
Note that $f$ is symmetric in $x$ and $y$.
Define $h(x,y)=f(x^2,y^2)$.
One easily computes that $h(x,0)=1+2x+3x^2+4x^3+5x^4+6x^5+7x^6+8x^7+9x^8+10x^9+\ld$.
These are the first terms of the Taylor series of $\inv{(1-x)^2}$, so we expect $h(x,0)$ and $h(0,y)$ to be equal to $\inv{(1-x)^2}$ and $\inv{(1-y)^2}$, respectively.
The `easiest' function satisfying this is $h(x,y)=\inv{(1-x-y)^2}$.
Computing the power series of both functions, we see that this indeed holds up to degree 10.
Hence we guess that $\Phi(r|x,y) = \inv{(1-\sqrt{x}-\sqrt{y})^2}$.
\end{ex}

\clearpage

\begin{table}
\centering
\caption{Formulas for $F_C\left(r,-r,\half,\ld,\half|z_1, \ld, z_n\right)$
\label{tab:F41_formula}}
\begin{tabular}{p{\eerstekollauricella}@{}p{\tweedekollauricella}}
\toprule
Function: & $F_C\left(r,-r,\half,\ld,\half|z_1, \ld, z_n\right)$ \\[6pt]
Horn equations ($n=2$): & $(\theta_x (\theta_x - \half) - x (\theta_x + \theta_y + r) (\theta_x + \theta_y - r)) F = 0$ \\[6pt]
 & $(\theta_y (\theta_y - \half) - y (\theta_x + \theta_y + r) (\theta_x + \theta_y - r)) F = 0$ \\[6pt]
$\HA$: & $\A = \{\ve_1, \ve_2, \ld, \ve_{n+2}, \ve_1+\ve_2-\ve_3, \ve_1+\ve_2-\ve_4, \ld, \ve_1+\ve_2-\ve_{n+2}\}$ \\[6pt] 
 & $\bbeta = (-r, r, -\half, \ld, -\half)$ \\[6pt] 
Rank: & $2^n$ \\[6pt]
Triangulation ($n=2$): & $\{\{1, 2, 3, 4\}, \{1, 2, 3, 6\}, \{1, 2, 4, 5\}, \{1, 2, 5, 6\}\}$ \\[6pt]
Basis ($n=2$): & $\Phi_1(r|x,y) = F_4\left(r,-r,\half,\half|x,y\right)$ \\[6pt] 
 & $\Phi_2(r|x,y) = \sqrt{y} F_4\left(r+\half, -r+\half, \half, \frac{3}{2}|x,y\right)$ \\[6pt] 
 & $\Phi_3(r|x,y) = \sqrt{x} F_4\left(r+\half, -r+\half, \frac{3}{2}, \half|x,y\right)$ \\[6pt] 
 & $\Phi_4(r|x,y) = \sqrt{x y} F_4\left(r+1,-r+1,\frac{3}{2}, \frac{3}{2} |x,y\right)$ \\[6pt]
Solution of~\eqref{eq:function_power} ($n=2$): & $\Phi = \Phi_1 + 2 i r \Phi_2 - 2 i r \Phi_3 + 4 r^2 \Phi_4$ \\[6pt]
Formulas: & $\Phi(r|\vect{z}) = f(\vect{z})^r$ \\[6pt]
 & $f(\vect{z}) = h(\vect{z}) + \sqrt{h(\vect{z})^2-1}$ \\[6pt]
 & $h(\vect{z}) = 1 - 2 (\sqrt{z_1}+\ld+\sqrt{z_n})^2$ \\
\bottomrule
\end{tabular}
\end{table}

\begin{table}
\centering
\caption{Formulas for $F_C\left(r,r+\half,\half,\ld,\half|z_1, \ld, z_n\right)$
\label{tab:F42_formula}}
\begin{tabular}{p{\eerstekollauricella}@{}p{\tweedekollauricella}}
\toprule
Function: & $F_C\left(r,r+\half,\half,\ld,\half|z_1, \ld, z_n\right)$ \\[6pt]
Horn equations ($n=2$): & $(\theta_x (\theta_x - \half) - x (\theta_x + \theta_y + r) (\theta_x + \theta_y + r + \half)) F = 0$ \\[6pt]
 & $(\theta_y (\theta_y - \half) - y (\theta_x + \theta_y + r) (\theta_x + \theta_y + r + \half)) F = 0$ \\[6pt]
$\HA$: & $\A = \{\ve_1, \ve_2, \ld, \ve_{n+2}, \ve_1+\ve_2-\ve_3, \ve_1+\ve_2-\ve_4, \ld, \ve_1+\ve_2-\ve_{n+2}\}$ \\[6pt] 
 & $\bbeta = (-r, -r-\half, -\half, \ld, -\half)$ \\[6pt] 
Rank: & $2^n$ \\[6pt]
Triangulation ($n=2$): & $\{\{1, 2, 3, 4\}, \{1, 2, 3, 6\}, \{1, 2, 4, 5\}, \{1, 2, 5, 6\}\}$ \\[6pt]
Basis ($n=2$): & $\Phi_1(r|x,y) = F_4\left(r,r+\half,\half,\half|x,y\right)$ \\[6pt] 
 & $\Phi_2(r|x,y) = \sqrt{y} F_4\left(r+\half, r+1, \half, \frac{3}{2}|x,y\right)$ \\[6pt] 
 & $\Phi_3(r|x,y) = \sqrt{x} F_4\left(r+\half, r+1, \frac{3}{2}, \half|x,y\right)$ \\[6pt] 
 & $\Phi_4(r|x,y) = \sqrt{x y} F_4\left(r+1,r+\frac{3}{2},\frac{3}{2}, \frac{3}{2} |x,y\right)$ \\[6pt]
Solution of~\eqref{eq:function_power} ($n=2$): & $\Phi = \Phi_1 + 2 r \Phi_2 + 2 r \Phi_3 + 2r(2r+1) \Phi_4$ \\[6pt]
Formulas: & $\Phi(r|\vect{z}) = f(\vect{z})^r$ \\[6pt]
 & $f(\vect{z}) = (\sqrt{z_1}+\ld+\sqrt{z_n}-1)^{-2}$ \\
\bottomrule
\end{tabular}
\end{table}

\begin{table}
\centering
\caption{Formulas for $F_C\left(r,r+\half,\half,\ld,\half,2r|z_1, \ld, z_n\right)$
\label{tab:F43_formula}}
\begin{tabular}{p{\eerstekollauricella}@{}p{\tweedekollauricella}}
\toprule
Function: & $F_C\left(r,r+\half,\half,\ld,\half,2r|z_1, \ld, z_n\right)$ \\[6pt]
Horn equations ($n=2$): & $(\theta_x (\theta_x - \half) - x (\theta_x + \theta_y + r) (\theta_x + \theta_y + r + \half)) F = 0$ \\[6pt]
 & $(\theta_y (\theta_y +2r-1) - y (\theta_x + \theta_y + r) (\theta_x + \theta_y + r + \half)) F = 0$ \\[6pt]
$\HA$: & $\A = \{\ve_1, \ve_2, \ld, \ve_{n+2}, \ve_1+\ve_2-\ve_3, \ve_1+\ve_2-\ve_4, \ld, \ve_1+\ve_2-\ve_{n+2}\}$ \\[6pt] 
 & $\bbeta = (-r, -r-\half, -\half, \ld, -\half, 2r-1)$ \\[6pt] 
Rank: & $2^n$ \\[6pt]
Triangulation ($n=2$): & $\{\{1, 2, 3, 4\}, \{1, 2, 3, 6\}, \{1, 2, 4, 5\}, \{1, 2, 5, 6\}\}$ \\[6pt]
Basis ($n=2$): & $\Phi_1(r|x,y) = F_4\left(r,r+\half,\half,2r|x,y\right)$ \\[6pt] 
 & $\Phi_2(r|x,y) = y^{1-2r} F_4\left(-r+1, -r+\frac{3}{2}, \half, 2-2r|x,y\right)$ \\[6pt] 
 & $\Phi_3(r|x,y) = \sqrt{x} F_4\left(r+\half, r+1, \frac{3}{2}, 2r|x,y\right)$ \\[6pt] 
 & $\Phi_4(r|x,y) = \sqrt{x} y^{1-2r} F_4\left(\frac{3}{2}-r,2-r,\frac{3}{2},2-2r |x,y\right)$ \\[6pt]
Solution of~\eqref{eq:function_power} ($n=2$): & $\Phi = \Phi_1 + 2 r \Phi_3$ \\[6pt]
Formulas: & $\Phi(r|\vect{z}) = f(\vect{z})^r g(\vect{z})$ \\[6pt]
 & $f(\vect{z}) = \frac{8 h(\vect{z})^2 - 4 z_n + 8 h(\vect{z}) \sqrt{h(\vect{z})^2-z_n}}{z_n^2}$ \\[6pt]
 & $g(\vect{z}) = \half - \frac{h(\vect{z})}{2 \sqrt{h(\vect{z})^2 - z_n}}$ \\[6pt]
 & $h(\vect{z}) = \sqrt{z_1}+\ld+\sqrt{z_{n-1}}-1$ \\
\bottomrule
\end{tabular}
\end{table}

\begin{table}
\centering
\caption{Formulas for $G_3(r,1-r|x,y)$
\label{tab:G3_formula}}
\begin{tabular}{p{\eerstekolhorn}@{}p{\tweedekolhorn}}
\toprule
Function: & $G_3(r,1-r|x,y)$ \\[6pt]
Horn equations: & $(\theta_x (-\theta_x + 2\theta_y + r) - x (2\theta_x - \theta_y - r + 1) (2\theta_x - \theta_y - r + 2)) F = 0$ \\[6pt]
 & $(\theta_y (2\theta_x - \theta_y + 1 - r) - y (-\theta_x + 2\theta_y + r) (-\theta_x + 2\theta_y + r + 1)) F = 0$ \\[6pt]
$\HA$: & $\A = \{\ve_1+\ve_2, \ve_2, -\ve_1+\ve_2, 2\ve_1+\ve_2\}$ \\[6pt] 
 & $\bbeta = (-r, -1)$ \\[6pt] 
Rank: & Horn system: 4; \ahypergeometric system: 3 \\[6pt]
Triangulation: & $\{\{1, 2\}, \{2, 3\}, \{1, 4\}\}$ \\[6pt]
Basis: & $\Phi_1(r|x,y) = G_3(r,1-r|x,y)$ \\[6pt] 
 & $\Phi_2(r|x,y) = x^r \sum \frac{\poch{r+1}{2m+3n}}{\poch{r+1}{m+2n} m! n!} (-x)^m (x^2 y)^n$ \\[6pt] 
 & $\Phi_3(r|x,y) = y^{1-r} \sum \frac{\poch{2-r}{3m+2n}}{\poch{2-r}{2m+n} m! n!} (x y^2)^m y^n$ \\[6pt] 
 & Horn system: $\Phi_4(r|x,y) = x^{\frac{r-2}{3}} y^{\frac{-r-1}{3}}$ \\[6pt]
Solution of~\eqref{eq:function_power}: & $\Phi = \Phi_1$ \\[6pt]
Formulas: & $\Phi(r|x,y) = f(x,y)^r \sqrt{\frac{g(x,y)}{\Delta(x,y)}}$ \\[6pt] 
 & $y f(x,y)^3 + f(x,y)^2 - f(x,y) + x = 0$ \\[6pt]
 & $g(x,y) = -3y^2 f(x,y)^2 - 2y f(x,y) + 4y + 1$ \\[6pt]
 & $\Delta(x,y) = 1+4x+4y+18x y - 27 x^2 y^2$ \\
\bottomrule
\end{tabular}
\end{table}

\begin{table}
\centering
\caption{Formulas for $H_4\left(r,-r,\half,\half|x,y\right)$
\label{tab:H41_formula}}
\begin{tabular}{p{\eerstekolhorn}@{}p{\tweedekolhorn}}
\toprule
Function: & $H_4\left(r,-r,\half,\half|x,y\right)$ \\[6pt]
Horn equations: & $(\theta_x (\theta_x - \half) - x (2\theta_x + \theta_y + r) (2\theta_x + \theta_y + r + 1)) F= 0$ \\[6pt]
 & $(\theta_y (\theta_y - \half) - y (2\theta_x + \theta_y + r) (\theta_y - r)) F = 0$ \\[6pt]
$\HA$: & $\A = \{\ve_1, \ve_2, \ve_3, \ve_4, 2\ve_1-\ve_3, \ve_1+\ve_2-\ve_4\}$ \\[6pt] 
 & $\bbeta = (-r, r, -\half, -\half)$ \\[6pt] 
Rank: & 4 \\[6pt]
Triangulation: & $\{\{1, 2, 3, 4\}, \{1, 2, 3, 6\}, \{1, 2, 4, 5\}, \{1, 2, 5, 6\}\}$ \\[6pt]
Basis: & $\Phi_1(r|x,y) = H_4\left(r,-r,\half,\half|x,y\right) $ \\[6pt] 
 & $\Phi_2(r|x,y) = \sqrt{y} H_4\left(r+\half,-r+\half,\half,\frac{3}{2}|x,y\right)$ \\[6pt] 
 & $\Phi_3(r|x,y) = \sqrt{x} H_4\left(r+1,-r,\frac{3}{2},\half|x,y\right)$ \\[6pt] 
 & $\Phi_4(r|x,y) = \sqrt{x y} H_4\left(r+\frac{3}{2},-r+\half,\frac{3}{2},\frac{3}{2}|x,y\right)$ \\[6pt]
Solution of~\eqref{eq:function_power}: & $\Phi = \Phi_1 - 2ir \Phi_2 + 2r \Phi_3 - 2 i r(2r+1) \Phi_4$ \\[6pt]
Formulas: & $\Phi(r|x,y) = f(x,y)^r$ \\[6pt] 
 & $f(x,y) = \frac{1-2\sqrt{x}+2y+2\sqrt{y(-1+2\sqrt{x}+y)}}{(1-2\sqrt{x})^2}$ \\
\bottomrule
\end{tabular}
\end{table}

\begin{table}
\centering
\caption{Formulas for $H_4\left(r,r+\half,\half,\half|x,y\right)$
\label{tab:H42_formula}}
\begin{tabular}{p{\eerstekolhorn}@{}p{\tweedekolhorn}}
\toprule
Function: & $H_4\left(r,r+\half,\half,\half|x,y\right)$ \\[6pt]
Horn equations: & $(\theta_x (\theta_x - \half) - x (2\theta_x + \theta_y + r) (2\theta_x + \theta_y + r + 1)) F = 0$ \\[6pt]
 & $(\theta_y (\theta_y - \half) - y (2\theta_x + \theta_y + r) (\theta_y + r + \half)) F = 0$ \\[6pt]
$\HA$: & $\A = \{\ve_1, \ve_2, \ve_3, \ve_4, 2\ve_1-\ve_3, \ve_1+\ve_2-\ve_4\}$ \\[6pt] 
 & $\bbeta = (-r, -r-\half, -\half, -\half)$ \\[6pt] 
Rank: & 4 \\[6pt]
Triangulation: & $\{\{1, 2, 3, 4\}, \{1, 2, 3, 6\}, \{1, 2, 4, 5\}, \{1, 2, 5, 6\}\}$ \\[6pt]
Basis: & $\Phi_1(r|x,y) = H_4\left(r,r+\half,\half,\half|x,y\right) $ \\[6pt] 
 & $\Phi_2(r|x,y) = \sqrt{y} H_4\left(r+\half,r+1,\half,\frac{3}{2}|x,y\right)$ \\[6pt] 
 & $\Phi_3(r|x,y) = \sqrt{x} H_4\left(r+1,r+\half,\frac{3}{2},\half|x,y\right)$ \\[6pt] 
 & $\Phi_4(r|x,y) = \sqrt{x y} H_4\left(r+\frac{3}{2},r+1,\frac{3}{2},\frac{3}{2}|x,y\right)$ \\[6pt]
Solution of~\eqref{eq:function_power}: & $\Phi = \Phi_1 - 2r \Phi_2 + 2r \Phi_3 - 2r(2r+1) \Phi_4$ \\[6pt]
Formulas: & $\Phi(r|x,y) = f(x,y)^r$ \\[6pt] 
 & $f(x,y) = \inv{(\sqrt{1-2\sqrt{x}}+\sqrt{y})^2}$ \\
\bottomrule
\end{tabular}
\end{table}

\begin{table}
\centering
\caption{Formulas for $H_4\left(r,r+\half,\half,2r|x,y\right)$
\label{tab:H43_formula}}
\begin{tabular}{p{\eerstekolhorn}@{}p{\tweedekolhorn}}
\toprule
Function: & $H_4\left(r,r+\half,\half,2r|x,y\right)$ \\[6pt]
Horn equations: & $(\theta_x (\theta_x - \half) - x (2\theta_x + \theta_y + r) (2\theta_x + \theta_y + r + 1)) F = 0$ \\[6pt]
 & $(\theta_y (\theta_y +2r-1) - y (2\theta_x + \theta_y + r) (\theta_y + r + \half)) F = 0$ \\[6pt]
$\HA$: & $\A = \{\ve_1, \ve_2, \ve_3, \ve_4, 2\ve_1-\ve_3, \ve_1+\ve_2-\ve_4\}$ \\[6pt] 
 & $\bbeta = (-r, -r-\half, -\half, 2r-1)$ \\[6pt] 
Rank: & 4 \\[6pt]
Triangulation: & $\{\{1, 2, 3, 4\}, \{1, 2, 3, 6\}, \{1, 2, 4, 5\}, \{1, 2, 5, 6\}\}$ \\[6pt]
Basis: & $\Phi_1(r|x,y) = H_4\left(r,r+\half,\half,2r|x,y\right)$ \\[6pt] 
 & $\Phi_2(r|x,y) = y^{1-2r} H_4\left(-r+1, -r+\frac{3}{2}, \half, -2r+2|x,y\right)$ \\[6pt] 
 & $\Phi_3(r|x,y) = \sqrt{x} H_4\left(r+1,r+\half,\frac{3}{2},2r|x,y\right)$ \\[6pt] 
 & $\Phi_4(r|x,y) = \sqrt{x} y^{1-2r} H_4\left(-r+2,-r+\frac{3}{2},\frac{3}{2},2-2r|x,y\right)$ \\[6pt]
Solution of~\eqref{eq:function_power}: & $\Phi = \Phi_1 + 2r \Phi_3$ \\[6pt]
Formulas: & $\Phi(r|x,y) = f(x,y)^r g(x,y)$ \\[6pt] 
 & $f(x,y) = \frac{-16\sqrt{x}-4y+8-4h(x,y)}{y^2}$ \\[6pt]
 & $g(x,y) = \half + \frac{1-2\sqrt{x}}{2h(x,y)}$ \\[6pt]
 & $h(x,y) = \sqrt{(2\sqrt{x}-1)(2\sqrt{x}+y-1)}$ \\
\bottomrule
\end{tabular}
\end{table}

\begin{table}
\centering
\caption{Formulas for $H_5(r,-r,\half|x,y)$
\label{tab:H5_formula}}
\begin{tabular}{p{\eerstekolhorn}@{}p{\tweedekolhorn}}
\toprule
Function: & $H_5(r,-r,\half|x,y)$ \\[6pt]
Horn equations: & $(\theta_x (-\theta_x + \theta_y - r) - x (2\theta_x + \theta_y + r) (2\theta_x + \theta_y + r + 1)) F = 0$ \\[6pt]
 & $(\theta_y (\theta_y - \half) - y (2\theta_x + \theta_y + r) (-\theta_x + \theta_y -r)) F = 0$ \\[6pt]
$\HA$: & $\A = \{\ve_1, \ve_2, \ve_3, 2\ve_1-\ve_2, \ve_1+\ve_2-\ve_3\}$  \\[6pt]
 & $\bbeta = (-r,r,-\half)$ \\[6pt] 
Rank: & 4 \\[6pt]
Triangulation: & $\{\{1, 2, 3, 4\}, \{1, 2, 3, 6\}, \{1, 2, 4, 5\}, \{1, 2, 5, 6\}\}$ \\[6pt]
Basis: & $\Phi_1(r|x,y) = H_5(r,-r,\half | x,y)$ \\[6pt] 
 & $\Phi_2(r|x,y) = \sqrt{y} H_5(r+\half,-r+\half,\frac{3}{2} | x,y)$ \\[6pt] 
 & $\Phi_3(r|x,y) = x^{-r} \sum \frac{\poch{-r}{2m+3n}}{\poch{\half}{n} \poch{1-r}{m+n} m! n!} (-x)^m (x y)^n $ \\[6pt] 
 & $\Phi_4(r|x,y) = x^{-r+\half} \sqrt{y} \sum \frac{\poch{\frac{3}{2}-r}{2m+3n}}{\poch{\frac{3}{2}-r}{m+n} \poch{\frac{3}{2}}{n} m! n!} (-x)^m (x y)^n$ \\[6pt]
Solution of~\eqref{eq:function_power}: & $\Phi = \Phi_1+ 2ir \Phi_2$ \\[6pt]
Formulas: & $\Phi(r|x,y) = f(x,y)^r$ \\[6pt] 
 & $x^2 f^4 + 2 x f^3 + (1-2x) f^2 + (4y-2) f + 1 = 0$ \\
\bottomrule
\end{tabular}
\end{table}

\clearpage

\bibliographystyle{mybibstyle5}
\bibliography{authors,journalsabbreviations,bibliography}

\end{document}